\documentclass[11pt,reqno]{amsart}
\usepackage[english]{babel}
\usepackage{graphicx}
\usepackage[usenames]{xcolor}
\usepackage[all]{xy}
\usepackage{amsfonts,dsfont,mathrsfs}
\usepackage{amsmath, amsthm, amssymb}
\usepackage{enumerate, url}

\numberwithin{equation}{section}

\newtheorem{thm}{Theorem}[section]

\newtheorem{prop}[thm]{Proposition}
\newtheorem{lem}[thm]{Lemma}
\theoremstyle{definition}
\newtheorem{dfn}[thm]{Definition}

 \newtheorem{claim}[thm]{Claim}
  \newtheorem{rmk}[thm]{Remark}
\newtheorem{ques}[thm]{Question}

\DeclareMathOperator{\Span}{span}

\newcommand{\bbR}{\ensuremath{\mathbb{R}}}
\newcommand{\bbQ}{\ensuremath{\mathbb{Q}}}

\newcommand{\bbZ}{\ensuremath{\mathbb{Z}}}

\newcommand{\bbP}{\ensuremath{\mathbb{P}}}
\newcommand{\bbH}{\ensuremath{\mathbb{H}}}
\newcommand{\bbE}{\ensuremath{\mathbb{E}}}
\newcommand{\bbA}{\ensuremath{\mathbb{A}}}

\newcommand{\iell}{\reflectbox{\rotatebox[origin=c]{180}{$\ell$}}}
\definecolor{aqua}{RGB}{61,227,244}

\date{\today}

\begin{document}

\title{The fundamental theorem of affine geometry on tori}

\author{Jacob Shulkin}

\author{Wouter van Limbeek}
\address{Department of Mathematics \\ 
                 University of Michigan \\
                 Ann Arbor, MI 48109}

\date{\today}

\begin{abstract} The classical Fundamental Theorem of Affine Geometry states that for $n\geq 2$, any bijection of $n$-dimensional Euclidean space that maps lines to lines (as sets) is given by an affine map. We consider an analogous characterization of affine automorphisms for compact quotients, and establish it for tori: A bijection of an $n$-dimensional torus ($n\geq 2$) is affine if and only if it maps lines to lines. \end{abstract}

\maketitle
\tableofcontents

\section{Introduction}
\label{sec:intro}

\subsection{Main result} A map $f:\bbR^n\to \bbR^n$ is \emph{affine} if there is an $n\times n$-matrix $A$ and $b\in\bbR^n$ such that $f(x)=Ax+b$ for all $x\in\bbR^n$. The classical Fundamental Theorem of Affine Geometry (hereafter abbreviated as FTAG), going back to Von Staudt's work in the 1840s \cite{ftag}, characterizes invertible affine maps of $\bbR^n$ for $n\geq 2$: Namely, if $f:\bbR^n\to \bbR^n$ is a bijection so that for any line $\ell$, the image $f(\ell)$ is also a line, then $f$ is affine. 

Since then, numerous generalizations and variations have been proven, both algebraic and geometric. For example, on the algebraic side, there is a version for for projective spaces (see e.g. \cite{ftpg}), general fields $k$, and for noninvertible maps (by Chubarev-Pinelis \cite{ftgnoninv}). On the geometric side, an analogue of FTAG holds in hyperbolic geomety (see e.g. \cite{fthg}), as well as in Lorentzian geometry, where the role of lines is played by lightcones (by Alexandrov \cite{ftlg}). 

All of these versions have in common that they are characterizations of self-maps of the model space of the corresponding geometry (e.g. $\bbE^n, \bbH^n$, or $\bbP^n$). On the other hand, if $X$ is some topological space, and $\mathcal{C}$ is a sufficiently rich family of curves on $X$, then one can expect rigidity for self-maps of $X$ preserving $\mathcal{C}$. In this light it seems reasonable to ask for the following generalized version of the FTAG:
	\begin{ques}[Generalized FTAG] Let $M$ be an affine manifold of dimension $>1$. Suppose that $f:M\to M$ is a bijection such that for any affine line $\ell$ in $M$, the image $f(\ell)$ is also an affine line (as a set). Is $f$ affine?\label{q:genftag}\end{ques}
	\begin{rmk} Here by \emph{affine manifold} and \emph{affine line}, we use the language of geometric structures: If $M$ is a smooth manifold of dimension $n$, then an \emph{affine structure} on $M$ is a covering of $M$ by $\bbR^n$-valued charts whose transition functions are locally affine. An \emph{affine line} in an affine manifold $M$ is a curve in $M$ that coincides with an affine line segment in affine charts. A map is \emph{affine} if in local affine coordinates, the map is affine. See \cite{geomstruct} for more information on affine manifolds and geometric structures. \label{rmk:geomstruct}\end{rmk}
	
	\begin{rmk} Of course one can formulate analogues of Question \ref{q:genftag} for the variations of FTAG cited above, e.g. with affine manifolds (affine lines) replaced with projective manifolds (projective lines) or hyperbolic manifolds (geodesics).\end{rmk}
	
	The only case in which Question \ref{q:genftag} known is the classical FTAG (i.e. $M=\bbA^n$ is affine $n$-space). Our main result is that Question \ref{q:genftag} has a positive answer for the standard affine torus:
	\begin{thm} Let $n\geq 2$ and let $T=\bbR^n\slash \bbZ^n$ denote the standard $n$-torus. Let $f:T\to T$ be any bijection that maps lines to lines (as sets). Then $f$ is affine. \label{thm:main}\end{thm}
	\begin{rmk} Note that $f$ is not assumed to be continuous! Therefore it is not possible to lift $f$ to a map of $\bbR^n$ and apply the classical FTAG. In addition, the proof of FTAG does not generalize to the torus setting: The classical proof starts by showing that $f$ maps midpoints to midpoints. Iteration of this property gives some version of continuity for $f$, which is crucial to the rest of the proof. The reason that $f$ preserves midpoints is that the midpoint of $P$ and $Q$ is the intersection of the diagonals of any parallellogram that has $PQ$ as a diagonal. This property is clearly preserved by $f$.
	
	There are several reasons this argument fails on the torus: First, any two lines may intersect multiple, even infinitely many, times. Hence there is no hope of characterizing the midpoint as the unique intersection of two diagonals. And second, any two points are joined by infinitely many distinct lines. Therefore there is no way to talk about the diagonals of a parallellogram.
	
	In fact, these two geometric differences (existence of multiple intersections and multiple lines between points) will play a crucial role in the proof of Theorem \ref{thm:main}.\end{rmk}
	
	\begin{rmk} The standard affine structure on the torus is the unique Euclidean structure. However, there are other affine structures on the torus, both complete and incomplete. We refer to \cite{geomstruct} for more information. We do not know whether a similar characterization of affine maps holds for these other, non-Euclidean structures.\end{rmk}
	
Finally, let us mention some related questions in (pseudo-)Riemannian geometry. If $M$ is a smooth manifold, then two metrics $g_1$ and $g_2$ on $M$ are called \emph{geodesically equivalent} if they have the same geodesics (as sets). Of course if $M$ is a product, then scaling any factor will not affect the geodesics. Are any two geodesically equivalent metrics isometric up to scaling on factors? 

This is of course false for spheres (any projective linear map preserves great circles), but Matveev essentially gave a positive answer for Riemannian manifolds with infinite fundamental group \cite{geodeq}: If $M$ admits two Riemannian metrics that are geodesically equivalent but not homothetic, and $\pi_1(M)$ is infinite, then $M$ supports a metric such that the universal cover of $M$ splits as a Riemannian product.

Of course Matveev's result makes no reference to maps that preserve geodesics. The related problem for maps has been considered with a regularity assumption, and has been called the ``Projective Lichnerowicz Conjecture" (PLC): First we say that a smooth map $f:M\rightarrow M$ of a closed (pseudo-)Riemannian manifold $M$ is \emph{affine} if $f$ preserves the Levi-Civita connection $\nabla$. Further $f$ is called \emph{projective} if $\nabla$ and $f^\ast\nabla$ have the same (unparametrized) geodesics. PLC then states that unless $M$ is covered by a round sphere, the group of affine transformations has finite index in the group of projective transformations.

For Riemannian manifolds, Zeghib has proven PLC \cite{zeghibPLC}. See also \cite{PLOC} for an earlier proof by Matveev of a variant of this conjecture. In view of these results, and Question \ref{q:genftag}, let us ask:

\begin{ques} Let $M$ be a closed nonpositively curved manifold of dimension $>1$ and let $\nabla$ be the Levi-Civita connection. Suppose $f:M\rightarrow M$ maps geodesics to geodesics (as sets). Is $f$ affine (i.e. smooth with $f^\ast \nabla=\nabla$)? \label{q:npc}\end{ques}

As far as we are aware, the answer to Question \ref{q:npc} is not known for any choice of $M$. Theorem \ref{thm:main} is of course a positive answer to Question \ref{q:npc} for the case that $M=T^n$ is a flat torus.
	
\subsection{Outline of the proof} Let $f:T^n\rightarrow T^n$ be a bijection preserving lines. The engine of the the proof is that $f$ preserves the number of intersections of two objects. A key observation is that affine geometry of $T^n$ allows for affine objects (lines, planes, etc.) to intersect in interesting ways (e.g. unlike in $\bbR^n$, lines can intersect multiple, even infinitely many, times). 

In Section \ref{sec:planes}, we give a characterization of rational subtori in terms of intersections with lines. This is used to prove that $f$ maps rational subtori to rational subtori. 

Then we start the proof of Theorem \ref{thm:main}. The proof is by induction on dimension, and the base case (i.e. dimension 2) is completed in Sections \ref{sec:q2dim} and \ref{sec:end2dim}. In Section \ref{sec:q2dim}, we start by recalling a characterization of the homology class of a rational line in terms of intersections with other lines. This allows us to associate to $f$ an induced map $A$ on $H_1(T^2)$, even though $f$ is not necessarily continuous. We regard $A$ as the linear model for $f$, and the rest of the section is devoted to proving $f=A$ (up to a translation) on the rational points $\bbQ^2\slash \bbZ^2$. In Section \ref{sec:end2dim}, we finish the proof by showing that $A^{-1}\circ f$ is given by a (group) isomorphism on each factor of $T^2$. We show this isomorphism lifts to a field automorphism of $\bbR$, and hence is trivial. This completes the proof of the 2-dimensional case

Finally in Section \ref{sec:ndim}, we use the base case and the fact that $f$ preserves rational subtori (proven in Section \ref{sec:planes}), to complete the proof in all dimensions.

\subsection{Notation} We will use the following notation for the rest of the paper. Let $n\geq 1$. Then $T=\bbR^n\slash\bbZ^n$ will be the standard affine $n$-torus. If $x\in\bbR^n$, then $[x]$ denotes the image of $x$ in $T$. Similarly, if $X\subseteq \bbR^n$ is any subset, then $[X]$ denotes the image of $X$ in $T$.

\subsection{Acknowledgments} WvL would like to thank Ralf Spatzier for interesting discussions on the geodesic equivalence problem, which sparked his interest in Question \ref{q:genftag}. In addition we would like to thank Kathryn Mann for conversations about the structure of homomorphisms of $S^1$. This work was completed as part of the REU program at the University of Michigan, for the duration of which JS was supported by NSF grant DMS-1045119. We would like to thank the organizers of the Michigan REU program for their efforts.

\section{Rational subtori are preserved}
\label{sec:planes}

\begin{dfn} Let $n\geq 1$ and $1\leq k\leq n$. Then a $k$-\emph{plane} $W$ in $T=\bbR^n\slash\bbZ^n$ is the image in $T$ of an affine $k$-dimensional subspace $V\subseteq \bbR^n$. Let $V_0$ be the translate of $V$ containing the origin. If $V_0$ is spanned by $V_0\cap \bbQ^n$, we say $W$ is \emph{rational}. In this case we also say $W$ is a \emph{rational subtorus} of dimension $k$. If $k=1$, we also say $W$ is a \emph{rational line}.\end{dfn}

The goal of this section is to show:

\begin{prop} Let $n\geq 2$ and $1\leq k\leq n$. Write $T=\bbR^n\slash\bbZ^n$ and let $f:T\to T$ be a bijection that maps lines to lines. Then the image under $f$ of any rational $k$-dimensional subtorus is again a rational $k$-dimensional subtorus.
\label{prop:qsub}
\end{prop}

For the rest of this section, we retain the notation of Proposition \ref{prop:qsub}. We start with the following criterion for subtori to be rational:

\begin{lem} Let $V$ be a plane in $T$. Then the following are equivalent:
\begin{enumerate}[(i)]
	\item $V$ is rational.
	\item $V$ is compact.
	\item Every rational line not contained in $V$ intersects $V$ at most finitely many times.
\end{enumerate}
\label{lem:qcrit}
\end{lem}

\begin{proof} The equivalence of (i) and (ii) is well-known. Let us prove (ii)$\implies$(iii). First suppose that $V$ is compact and let $\ell$ be any rational line not contained in $V$. If $\ell\cap V=\emptyset$, we are done. If $v_0\in \ell\cap V$, we can translate $\ell$ and $V$ by $-x_0$, so that we can assume without loss of generality that $\ell$ and $V$ intersect at 0. Then $\ell\cap V$ is a compact subgroup of $T$ and not equal to $\ell$. Because $\ell$ is 1-dimensional, any proper compact subgroup of $\ell$ is finite. This proves $\ell$ and $V$ intersect only finitely many times.

Finally we prove (iii)$\implies$(ii). Suppose that $V$ is not compact. Then $\overline{V}$ is a compact torus foliated by parallel copies of $V$. Let
	$$\psi:U\rightarrow \bbR^k\times \bbR^{n-k}$$
be a chart near $0\in T$ such that the slices $\bbR^k\times \{y\}, \, y\in \bbR^{n-k}$, correspond to the local leaves of the foliation. Since $V$ is dense in $\overline{V}$, there are infinitely many values of $y$ such that $\bbR^k\times y\subseteq \psi(V\cap U)$.

Now choose a rational line $\ell$ in $\overline{V}$ that is not contained in $V$ but with $0\in\ell$. If the neighborhood $U$ above is chosen sufficiently small, then $\psi(\ell\cap U)$ intersects all leaves $\bbR^k\times y$. Since there are infinitely many values of $y$ such that $\bbR^k\times y\subseteq \psi(V)$, it follows that $\ell\cap V$ is infinite. \end{proof}

Recall that we are trying to show that the image under $f$ of any rational subtorus is again a rational subtorus. We first show the image is a plane.

\begin{claim} Let $S$ be any rational subtorus of dimension $k$. Then $f(S)$ is a $k$-plane.
\label{claim:planes1} 
\end{claim}

\begin{proof}
We induct on $k=\dim(S)$. The base case $k=1$ is just the assertion that $f$ maps lines to lines.

\par
Suppose now $k\geq 2$ and that the claim holds for $l$-planes where $l<k$, and let $S$ be a rational $k$-dimensional subtorus. Choose some $x_0\in S$ and  let $S_0:=S-x_0$ be the translate of $S$ passing through 0. Let $V_0\subseteq \bbR^n$ be the subspace that projects to $S_0$. Since $S$ is rational, we can choose a basis $v_1,\dots,v_k$ of $V$ with $v_i\in\bbQ^n$ for every $1\leq i\leq k$.

To use the inductive hypothesis, consider the $(k-1)$-plane 
	$$V_{1}=\Span\{v_{1},...,v_{k-1}\}$$ 
and similarly $V_{2}=\Span\{v_{2},...,v_{k}\}$. For $i=1,2$, set $S_i:=[V_i+x_0]$. Since $v_i$ are rational vectors for each $i$, we clearly have that $S_1$ and $S_2$ are rational $(k-1)$-dimensional subtori. Also let $S_{12}:=S_1\cap S_2$ be the intersection, which is a $(k-2)$-dimensional rational subtorus. The inductive hypothesis implies that $f(S_1)$ and $f(S_2)$ and $f(S_{12})$ are all rational subtori of $T$. 

Since $f(S_1)$ and $f(S_2)$ are $(k-1)$-planes that intersect in a $(k-2)$-plane, they span a $k$-dimensional plane. More precisely, let $W_1$ (resp. $W_2$) be the subspace of $\bbR^n$ that projects to $f(S_1)-f(x_0)$ (resp. $f(S_2)-f(x_0)$). Then $W_1$ and $W_2$ are $(k-1)$-dimensional subspaces of $\bbR^n$ that intersect in a $(k-2)$-dimensional subspace, and hence span a $k$-dimensional subspace $W_0$. Then the $k$-plane $W:=[W_0]+f(x_0)$ contains both $f(S_1)$ and $f(S_2)$.

We claim that $f(S)=W$. We start by showing the inclusion $f(S)\subseteq W$. Let $x\in S$. If $x\in S_1\cup S_2$, then clearly $f(x)\in W$, so we will assume that $x\in S$ but $x\notin S_1\cup S_2$.

	\begin{figure}[h]
		\centering
		\includegraphics[height=3in]{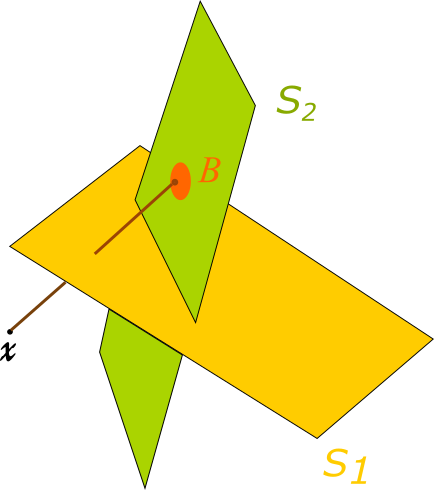}
		\caption{A ball $B\subseteq S_2$ such that lines from $x$ to $B$ meet $S_1$.}
		\label{fig:isec}
	\end{figure}

Since $S_1$ and $S_2$ are closed codimension 1 submanifolds of $S$, there is an open ball $B\subseteq S_2$ such that for any $y\in B$, there is a line joining $x$ and $y$ that intersects $S_1$ (see Figure \ref{fig:isec}). Set $Y:=f(B)$. Since $f$ sends lines to lines and preserves intersections, $f(x)$ is a point with the property that for any $y\in Y$, there is a line $\ell_y$ joining $f(x)$ and $y$ that intersects $f(S_1)$. We claim that this implies that $f(x)\in W$.

Write $U_1:=f(S_1)-f(x_0)$ for the translate of $f(S_1)$ that passes through the origin. We can regard $U_1$ as a subgroup of $T$, and consider the projection
	$$\pi: T\rightarrow T\slash U_1.$$
The image $\pi(W)$ of $W$ is a line because $f(S_1)$ has codimension 1 in $W$. Further $\pi(f(x))$ is a point with the property that for any $y\in Y$, the line $\pi(\ell_y)$ joins the point $\pi(f(S_1))=\pi(f(x_0))$ and $\pi(f(x))$, and intersects $\pi(W)$ at the point $\pi(y)$.

Note that there are only countably many lines from $\pi(f(x_0))$ to $\pi(f(x))$, and unless one of them is contained in $\pi(W)$, each one has at most countably many intersections with the line $\pi(W)$. However, since $\pi(Y)$ is uncountable, we conclude that not every point $\pi(y)$ can lie on a line that passes through both $\pi(f(x_0))$ and $\pi(f(x))$, unless $\pi(f(x))\in \pi(W)$. Therefore we must have that $\pi(f(x))\in\pi(W)$, so that $f(x)\in W$. This proves that $f(S)\subseteq W$.

\par 
To establish the reverse inclusion, we just apply the same argument to $f^{-1}$. The above argument then yields that $f^{-1}(W)\subseteq S$. Applying $f$ gives $W\subseteq f(S)$, as desired. \end{proof}

Actually the above proof also shows the following more technical statement, which basically states that linearly independent lines are mapped to linearly independent lines. We will not need this until Section \ref{sec:ndim}, but it is most convenient to state it here.

%
%
%
%
%
%
%
%

\begin{lem} Let $S=[V]$ be a rational subtorus containing 0, where $V\subseteq \bbR^n$ is a subspace. For $1\leq i\leq k:=\dim(S)$, let $v_i\in\bbQ^n$ such that $v_1,\dots,v_k$ is a basis for $V$. Set $\ell_i:=[\bbR v_i]$ and choose $w_i\in \bbQ^n$ such that $f(\ell_i)=[\bbR w_i]$. Then 
	$$f(S)=[\Span(w_1,\dots,w_k)].$$ 
	\label{lem:span} \end{lem}
	\vspace{-.5cm}
\begin{proof} Recall that (with the notation of the proof of Claim \ref{claim:planes1}), we have
	$$W=\Span(w_1,\dots,w_k),$$
	and we have shown $f(S)=[W]$, as desired.
\end{proof}

It remains to show that if $S$ is a rational subtorus, then $f(S)$ is also rational. We first show this for $S$ of codimension 1.

\begin{claim}
\label{claim:qcodim1}
Let $S \subseteq T$ be a rational codimension 1 subtorus. Then $f(S)$ is also rational. \end{claim}

\begin{proof} Let $S\subseteq T$ be a codimension 1 rational subtorus and let $\ell$ be any rational line not contained in $S$. By Lemma \ref{lem:qcrit} applied to $S$, we see that $\ell\cap S$ is finite. Then $f(\ell)$ and $f(S)$ also intersect only finitely many times. 

Suppose now that $f(S)$ is not rational. Then $\overline{f(S)}$ is a connected compact torus properly containing the codimension 1 plane $f(S)$, and therefore $\overline{f(S)}=T$, i.e. $f(S)$ is dense in $T$. But if $f(S)$ is dense, then it intersects any line that is not parallel to $S$ infinitely many times. We know that $f(\ell)$ is not parallel to $f(S)$, because $f(\ell)$ and $f(S)$ intersect at least once. On the other hand, $f(\ell)$ and $f(S)$ intersect finitely many times. This is a contradiction. \end{proof}

Finally we can finish the proof of Proposition \ref{prop:qsub}.

\begin{proof}[Proof of Proposition \ref{prop:qsub}] First note that if $S_1$ and $S_2$ are rational subtori, then any component of $S_1\cap S_2$ is also rational (e.g. by using that rationality is equivalent to compactness). 

Now let $S$ be any rational subtorus of codimension $l$. Then we can choose $l$ rational codimension 1 subtori $S_1,\dots,S_l$ such that $S$ is a component of $\cap_i S_i$. Since $f$ is a bijection, we have
	$$f(S)\subseteq f\left(\bigcap_{1\leq i\leq l} S_i\right)=\bigcap_{1\leq i\leq l} f(S_i)$$
and by Claim \ref{claim:qcodim1}, we know that $f(S_i)$ are rational. Therefore $f(S)$ is a codimension $l$-plane contained in $\cap_i f(S_i)$. The components of $\cap_i f(S_i)$ have codimension $l$, so we must have that $f(S)$ is a component, and hence is rational.\end{proof}

\section{The 2-dimensional case: Affinity on rational points}
\label{sec:q2dim}

The goal of this section is to prove a rational version Theorem \ref{thm:main} in the two-dimensional case. More precisely, under the assumptions of Theorem \ref{thm:main}, we will prove that there is a linear automorphism $A$ of $T^2$ such that, up to a translation of $T^2$, we have $f=A$ on $\bbQ^2\slash\bbZ^2$. In the next section, we will complete the proof of the two-dimensional case by proving that, up to a translation, $f=A$ on all of $T^2$. 

We start by recalling the following elementary computation of the number of intersections of a pair of rational lines.

\begin{prop} Let $\ell_1$ and $\ell_2$ be two affine rational lines in the torus. For $i=1,2$, let $v_i \in \bbZ^2$ be a primitive tangent vector to the translate of $\ell_i$ passing through $0\in T^2$. Then the number of intersections of $\ell_1$ and $\ell_2$ is given by
	$$|\ell_1\cap \ell_2|= \left| \det\begin{pmatrix} 
| & | \\
v_1 & v_2 \\
| & |
 \end{pmatrix}\right|.$$
 \label{prop:inr}
 \end{prop}
%
%
%

We now turn towards proving Theorem \ref{thm:main} in the two-dimensional case. For the rest of this section, suppose $f: T^2\to T^2$ is a bijection that maps lines to lines. Also let us fix the following notation: For $i=1,2$, set $\ell_i:=\begin{bmatrix} \bbR e_i\end{bmatrix}$. We first make some initial reductions: By replacing $f$ with 
	$$x\mapsto f(x)-f(0),$$ 
we can assume that $f(0)=0$. For the next reduction, we need the following claim.

\begin{claim} There is a linear automorphism $A:T^2\rightarrow T^2$ such that $A \ell_i = f(\ell_i)$.\label{claim:reduct}\end{claim}
\begin{proof} Since $f(0)=0$, the lines $f(\ell_i)$ pass through 0. In addition, because $\ell_i$ are rational, so are $f(\ell_i)$ (see Proposition \ref{prop:qsub}). Therefore there are coprime integers $p_i$ and $q_i$ such that 
	$$f(\ell_i)=\begin{bmatrix}\bbR \begin{pmatrix} p_i \\ q_i\end{pmatrix}\end{bmatrix}.$$ 
Note that $\ell_1$ and $\ell_2$ intersect exactly once, and hence so do $f(\ell_1)$ and $f(\ell_2)$. By Proposition \ref{prop:inr}, the number of intersections is also given by
	$$|f(\ell_1)\cap f(\ell_2)|=\left| \det\begin{pmatrix} p_1 & p_2 \\ q_1 & q_2\end{pmatrix}\right|,$$
so the linear transformation $A:T^2\rightarrow T^2$ with matrix
	$$A=\begin{pmatrix} p_1 & p_2 \\ q_1 & q_2 \end{pmatrix}$$
is an automorphism, and clearly satisfies $A\ell_i = f(\ell_i)$.\end{proof}

Let $A$ be as in Claim \ref{claim:reduct}. Then by replacing $f$ by $A^{-1}\circ f$, we can assume that $f(\ell_i)=\ell_i$ for $i=1,2$. Note that because each $v_i$ is only unique up to sign, the matrix $A$ is not canonically associated to $f$. Therefore we cannot expect that $A^{-1}\circ f=\text{id}$, and we have to make one further reduction to deal with the ambiguity in the definition of $A$. Let us introduce the following notation: For a rational line $\ell=\begin{bmatrix}\bbR\begin{pmatrix} p \\ q\end{pmatrix}\end{bmatrix}$, let \iell \, be the line obtained by reflecting $\ell$ in the $x$-axis, i.e. \iell\,=$\begin{bmatrix} \bbR \begin{pmatrix} p \\ -q\end{pmatrix} \end{bmatrix}$. 
 \begin{claim}
\label{claim:alt}
 Let $\ell$ be a rational line in $T^2$ passing through the origin. Then $f(\ell)=\ell$ or $f(\ell)=\iell$.
  \end{claim}

\begin{proof} Let $p,q$ be coprime integers such that $\ell=\begin{bmatrix} \bbR\begin{pmatrix} p \\ q\end{pmatrix} \end{bmatrix}$. Also choose coprime integers $r,s$ such that $f(\ell)=\begin{bmatrix}\bbR\begin{pmatrix} r\\ s\end{pmatrix}\end{bmatrix}$. We can compute $p,q,r,s$ as suitable intersection numbers. Indeed, we have
	$$|p|=\left| \det\begin{pmatrix} 
p & 0 \\ q & 1
 \end{pmatrix}\right| = |\ell\cap \ell_2|.$$
 
 Since $f(\ell_2)=\ell_2$ and $f$ preserves the number of intersections of a pair of lines, we see that
 	$$|p|=|\ell\cap \ell_2|=|f(\ell)\cap \ell_2| = |r|,$$
and similarly $|q|=|s|$. Therefore we find that
	$$\bbR \begin{pmatrix} p \\ q\end{pmatrix} = \bbR \begin{pmatrix} r \\ s \end{pmatrix} \hspace{1 cm} \text{or} \hspace{1 cm} \bbR \begin{pmatrix} p \\ -q\end{pmatrix} = \bbR\begin{pmatrix} r \\ s \end{pmatrix},$$
	which exactly corresponds to $\ell=f(\ell)$ or $\iell=f(\ell)$.\end{proof}
We will now show that whichever of the two alternatives of Claim \ref{claim:alt} occurs does not depend on the line $\ell$ chosen.	
\begin{claim}
\label{claim:onlyone} Let $\ell$ be a rational line passing through the origin that is neither horizontal nor vertical (i.e. $\ell\neq \ell_i$ for $i=1,2$). Suppose that $f(\ell)=\ell$. Then for any rational line $m$ passing through the origin, we have $f(m)=m$.
\end{claim}

\begin{proof} Let $p,q$ coprime integers such that $\ell=\begin{bmatrix}\bbR\begin{pmatrix} p \\ q\end{pmatrix}\end{bmatrix}$. Since $\ell$ is neither vertical nor horizontal, we know that $p$ and $q$ are nonzero.

Now let $m$ be any other rational line, and choose coprime integers $r,s$ such that $m=\begin{bmatrix}\bbR \begin{pmatrix} r\\ s\end{pmatrix}\end{bmatrix}$. The number of intersections of $\ell$ and $m$ is given by 
	$$|\ell\cap m| = \left| \det \begin{pmatrix} p & r \\ q & s\end{pmatrix}\right| = |ps-qr|.$$
We will argue by contradiction, so suppose that $f(m)\neq m$. Using the other alternative of Claim \ref{claim:alt} for $m$, we have that 
	$$|f(\ell)\cap f(m)|= \left| \det \begin{pmatrix} p & r \\ q & -s\end{pmatrix}\right| = |ps+qr|.$$
Since $f$ preserves the number of intersections, we therefore have
	$$|ps-qr|=|ps+qr|$$
which can only happen if $ps=0$ or $qr=0$. Since we know that $p$ and $q$ are nonzero, we must have $r=0$ or $s=0$. This means that $m$ is horizontal or vertical,  but in those cases the alternatives provided by Claim \ref{claim:alt} coincide. \end{proof}

We can now make the final reduction: If $f(\ell)=\ell$ for any rational line $\ell$, then we leave $f$ as is. In the other case, i.e. $f(\ell)=\iell$ for any rational line $\ell$, we replace $f$ by
	$$\begin{pmatrix} 1 & 0 \\ 0 & -1 \end{pmatrix} \circ f,$$
after which we have $f(\ell)=\ell$ for any rational line $\ell$ passing through the origin. After making these initial reductions, our goal is now to prove that $f=\text{id}$ on the rational points $\bbQ^2\slash \bbZ^2$. We start with the following easy observation:

\begin{claim} \label{claim:||} If $\ell$ and $m$ are parallel lines in $T$, then $f(\ell)$ and $f(m)$ are also parallel.\end{claim}
\begin{proof} Since $\dim(T)=2$, two lines in $T$ are parallel if and only if they do not intersect. This property is obviously preserved by $f$.\end{proof}
Combining our reductions on $f$ so far with Claim \ref{claim:||} yields the following quite useful property for $f$: Given any rational line $\ell$, the line $f(\ell)$ is parallel to $\ell$. However, we still do not know the behavior of $f$ in the direction of $\ell$ or transverse to $\ell$. In particular we do not know
	\begin{itemize}
		\item whether or not $f$ leaves invariant every line $\ell$; we only know this for $\ell$ passing through the origin, and
		\item whether or not $f=\text{id}$ on rational lines $\ell$ passing through the origin.
	\end{itemize}
We can now prove:
\begin{thm}
\label{theorem:rationals}
For any rational point $x\in \mathbb{Q}^2/\mathbb{Z}^2 $ we have $f(x)=x$.
\end{thm}

\begin{proof} For $n>2$, set
	$$G_{n}=\begin{bmatrix} \dfrac{1}{n} \, \bbZ^2\end{bmatrix}.$$
Clearly we have $\bbQ^2\slash\bbZ^2=\cup_n G_n$, so it suffices to show that $f=\text{id}$ on $G_n$ for every $n> 2$. We purposefully exclude the case $n=2$ because $G_2$ only consists of 4 points, and this is not enough for the argument below. However, since we have $G_2\subseteq G_6$, this does not cause any problems. Fix $n\geq 2$.

Set $Q:=\begin{bmatrix} \dfrac{1}{n}\begin{pmatrix} 1 \\ 1\end{pmatrix}\end{bmatrix}.$

\begin{claim} If $f(Q)=Q$, then $f=\text{id}$ on $G_n$. \label{claim:fqq} \end{claim}
\begin{proof} Suppose that $x=\begin{bmatrix} x_1 \\ y_1 \end{bmatrix}$ is any point of $G_n$ with $f(x)\neq x$. We can choose $x_1$ and $y_1$ that lie in $[0,1)$. We will show that $f(Q)\neq Q$. 

Let $x_2, y_2\in [0,1)$ be such that $f(x)=\begin{bmatrix} x_2 \\ y_2\end{bmatrix}$. Then there are two cases: Either $x_1\neq x_2$ or $y_1\neq y_2$. Clearly the entire setup is symmetric in the two coordinates, so we will just consider the case $x_1\neq x_2$. See Figure \ref{fig:grids} for an illustration of the various lines introduced below.

For $a\in[0,1)$, let $v_a$ (resp. $h_a$) denote the vertical (resp. horizontal) line $\begin{bmatrix} x=a\end{bmatrix}$ (resp. $\begin{bmatrix} y=a\end{bmatrix}$). Since $f$ maps parallel lines to parallel lines by Claim \ref{claim:||} and $f(\ell_i)=\ell_i$ by assumption, we know that $f$ maps vertical (resp. horizontal) lines to vertical (resp. horizontal) lines. In particular we have that $f(v_{x_1})=v_{x_2}$ and $f(h_{y_1})=h_{y_2}$. We will show that $f(Q)\neq Q$ by showing that $f(h_{1/n}) \neq h_{1/n}$.

\begin{figure}[h]
		\centering
		\includegraphics[height=3in]{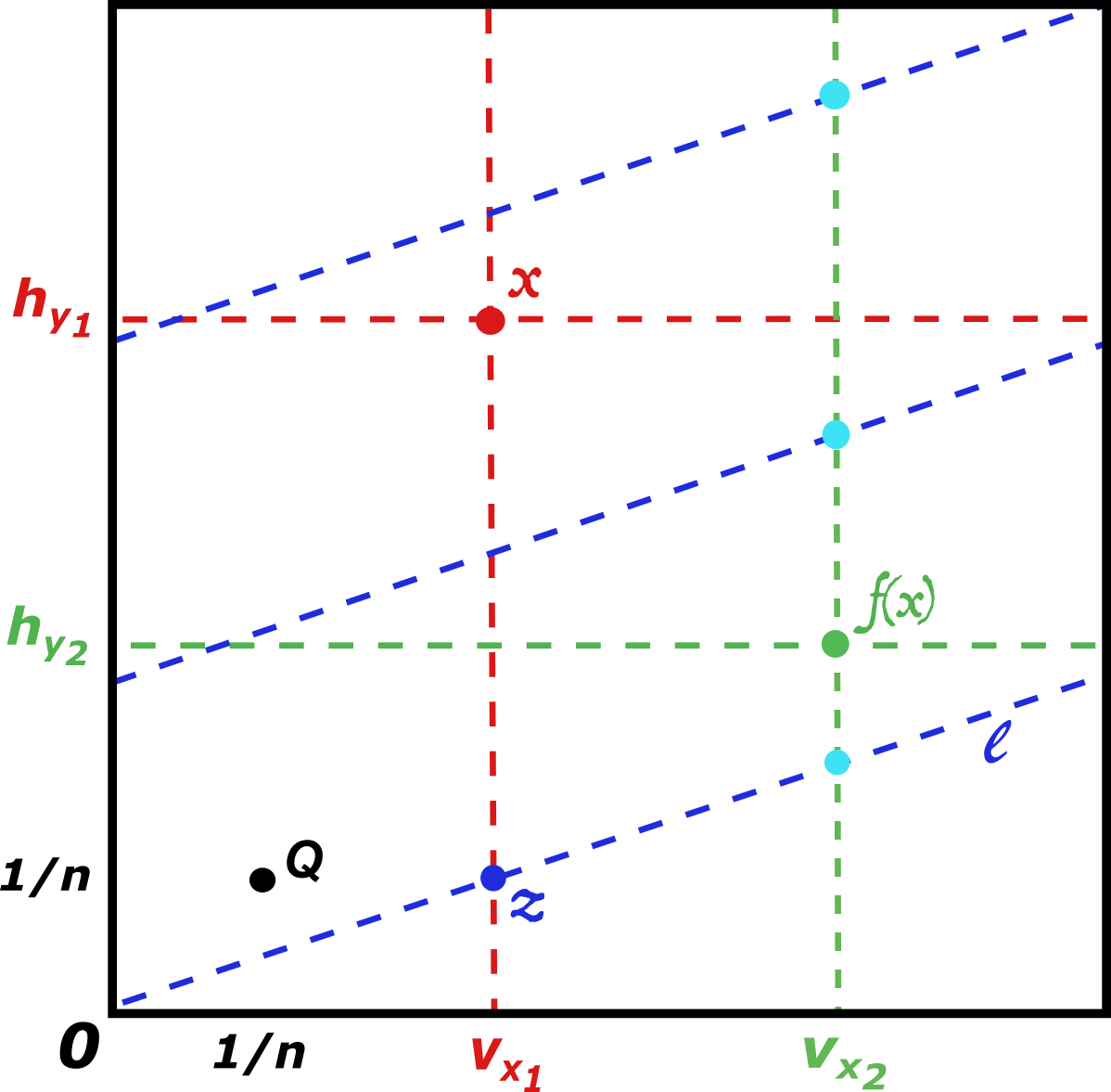}
		\caption{Illustration of the proof of Claim \ref{claim:fqq}. The points in \textcolor{aqua}{aqua} are the possibilities for $f(z)$. The proof amounts to showing all of these lie at different heights than $z$ does.}
		\label{fig:grids}
	\end{figure}

Consider the point $z:=\begin{bmatrix}  x_1 , \dfrac{1}{n} \end{bmatrix}\in T^2$. Since $x \in G_n$, there is an integer $0\leq a<n$ such that $x_1 = \dfrac{a}{n}$. First note that $a>0$: For if $a=0$, then $x\in \ell_2$. Since $f(\ell_2)=\ell_2$, we see that $x_2=0$ as well. This contradicts the initial assumption that $x_1\neq x_2$.

Since $a>0$, $z$ lies on the line $\ell$ of slope $\dfrac{1}{a}$ going through the origin. By Claim \ref{claim:onlyone} we have $f(\ell)=\ell$, so $f(z)\in\ell$. 

In addition, since $z\in v_{x_1}$, we have $f(z)\in v_{x_2}$. An easy computation yields
	$$\ell\cap v_{x_2} = \begin{bmatrix} x_2 \\ \frac{x_2}{a} \end{bmatrix} + \frac{1}{a}\begin{bmatrix} 0 \\ \bbZ\end{bmatrix}.$$
Therefore we automatically have $f(h_{\frac{1}{n}})\neq h_{\frac{1}{n}}$ unless 
	$$\frac{1}{a} x_2 + \frac{k}{a} \equiv \frac{1}{n} \mod \bbZ$$
for some integer $k$ with $0\leq k<a$. Solving for $k$ yields
	$$k\equiv \frac{a}{n}-x_2 \mod a\bbZ.$$
Since $0\leq k<a$ and $0<a<n$ and $0\leq x_2<1$, we see that the only option is that $k=0$. Hence $x_2\equiv \frac{a}{n}\equiv x_1\mod \bbZ$, which contradicts our initial assumption that $x_1\neq x_2$.
\end{proof}

It remains to show that $f(Q)=Q$. Consider the points $P, R, S$ as shown in Figure \ref{fig:quad}. Let us give a brief sketch of the idea so that the subsequent algebraic manipulations are more transparent. 

\begin{figure}[h]
		\centering
		\includegraphics[height=3in]{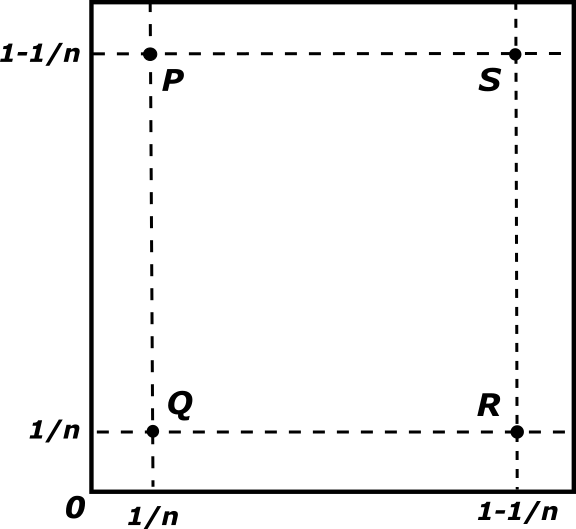}
		\caption{Quadrilateral in $T^2$ with vertices $P,Q,R,S$.}
		\label{fig:quad}
	\end{figure}

Because $P$ and $Q$ lie on a vertical line, and $f$ maps vertical lines to vertical lines, $f(P)$ and $f(Q)$ have to lie on the same vertical line. Because $P$ lies on the line of slope $n-1$ through the origin, $f(P)$ also has to lie on this line. Therefore however much $P$ and $f(P)$ differ in the horizontal direction determines how much they differ in the vertical direction. Finally since $P$ and $S$ lie on a horizontal line, the vertical difference between $P$ and $f(P)$ equals the vertical difference between $S$ and $f(S)$. The upshot of this is that any difference between $Q$ and $f(Q)$ translates into information about the vertical difference between $S$ and $f(S)$.

The same reasoning with $P$ replaced by $R$ yields information about the horizontal difference between $S$ and $f(S)$. Finally because $S$ lies on a line of slope 1 through the origin, any horizontal difference between $S$ and $f(S)$ matches the vertical difference. Therefore the information gained about the difference between $S$ and $f(S)$ using the two methods (once using $P$ and once using $R$) has to match. We will see that this forces $Q=f(Q)$.

Let us now carry out the calculations. Choose $\delta \in \left(\dfrac{-1}{n},\dfrac{n-1}{n}\right)$ such that
	$$f(v_{\frac{1}{n}})=v_{\frac{1}{n}}+\begin{bmatrix} \delta \\ 0\end{bmatrix},$$
so that $f(P)$ and $f(Q)$ have first coordinate $\frac{1}{n}+\delta$.  Note that the boundary cases $\delta=-\frac{1}{n}$ and $\delta=\frac{n-1}{n}$ are impossible since these correspond to $f(v_{\frac{1}{n}})=\ell_2$. Since $Q$ lies on the line of slope $1$ through the origin, we have
	\begin{equation} f(Q)=
	\begingroup
		\renewcommand*{\arraystretch}{2}
		\begin{bmatrix} \dfrac{1}{n}+\delta \\ \dfrac{1}{n}+\delta\end{bmatrix}
	\endgroup
 	\end{equation}
Since $P$ lies on the line of slope $n-1$ through the origin, there is some integer $k_P$ such that
	\begin{equation} f(P)=
		\begingroup
			\renewcommand*{\arraystretch}{2}
			\begin{bmatrix} \dfrac{1}{n}+\delta \\ \dfrac{n-1}{n}+\delta(n-1)-k_P\end{bmatrix}.
		\endgroup
		\label{eq:fp} 
	\end{equation}
Here we can choose $k_P$ such that 
	$$0\leq \frac{n-1}{n}+\delta(n-1)-k_P < 1.$$
Using that $R$ lies on the line of slope $\frac{1}{n-1}$ through the origin, we similarly find an integer $k_R$ such that
	$$f(R)=
		\begingroup
			\renewcommand*{\arraystretch}{2}
			\begin{bmatrix} \dfrac{n-1}{n}+\dfrac{\delta}{n-1}-k_R \\ \dfrac{1}{n}+\delta \end{bmatrix}.
		\endgroup
	$$
Finally, using that $S$ lies on the same horizontal line as $P$, and on the same vertical line as $R$, we find
	\begin{equation} f(S)=
		\begingroup
			\renewcommand*{\arraystretch}{2}
			\begin{bmatrix} \dfrac{n-1}{n}+\dfrac{\delta}{n-1}-k_R \\ \dfrac{n-1}{n}+\delta(n-1)-k_P\end{bmatrix}.
		\endgroup
		\label{eq:fs} 
	\end{equation}
Since $S$ lies on the line of slope 1 through the origin, so does $f(S)$. Setting the two coordinates of $f(S)$ given by Equation \ref{eq:fs} equal to each other, we find (after some simple algebraic manipulations):
	\begin{equation} k_P - k_R = \delta\left(n-1-\dfrac{1}{n-1}\right). \label{eq:rel} \end{equation}
We can obtain one more equation relating $k_P$ and $\delta$ by noting that $P$ lies on the line of slope $-1$ through the origin, and hence so must $f(P)$. This means that $f(P)$ is of the form $(x,1-x)$. Using Equation \ref{eq:fp}, this gives after some algebraic manipulations:
	\begin{equation} k_P=n\delta \label{eq:kpd}\end{equation}
We can use Equation \ref{eq:kpd} to eliminate $k_P$ from Equation \ref{eq:rel} to obtain:
	\begin{equation} k_R = \delta\left(1+\frac{1}{n-1}\right). \label{eq:krd}\end{equation}
Using that
	$$-\frac{1}{n}<\delta<\frac{n-1}{n},$$
and $n>2$ one easily sees that
	$$-1<k_R<1.$$
Since $k_R$ is also an integer, this forces $k_R=0$. Combining this with Equation \ref{eq:krd}, we see that $\delta=0$. Therefore $f(Q)=Q$, as desired.\end{proof}
This finishes the proof that $f$ is affine on $\bbQ^2\slash\bbZ^2$. In the next section, we will promote this to the entire torus.

\section{The 2-dimensional case: End of the proof}
\label{sec:end2dim}

At this point we know that the map $f:T^2\to T^2$ is a bijection sending lines to lines, with the additional properties that 
	\begin{itemize}
		\item $f=\text{id}$ on $\bbQ^2\slash\bbZ^2$,
		\item if $\ell$ is a line with rational slope $\alpha$, then $f(\ell)$ is also a line with slope $\alpha$, and
		\item if $\ell$ has irrational slope, then so does $f(\ell)$.
	\end{itemize}
In this section we will prove that $f=\text{id}$ on all of $T^2$. We start with the following observation. Write $T^2=S^1\times S^1$ and for $x,y\in S^1$, write
	$$f(x,y)=(f_1(x,y),f_2(x,y)).$$
Since $f$ maps vertical lines to vertical lines, the value of $f_1(x,y)$ does not depend on $y$. Similarly, the value of $f_2(x,y)$ does not depend on $x$. Hence for $i=1,2$, there are functions $f_i:S^1\to S^1$ such that
	$$f(x,y)=(f_1(x),f_2(y)).$$
We actually have $f_1=f_2$: Indeed, for any $x\in S^1$, consider the point $p:=(x,x)\in T^2$. Since $p$ lies on the line of slope 1 through the origin, so does $f(p)=(f_1(x),f_2(x))$. Hence $f_1=f_2$. We will write $\sigma:S^1\to S^1$ for this map, so that
	$$f(x,y)=(\sigma(x),\sigma(y))$$
for any $x,y\in S^1$. 

Let us now outline the proof that $f=\text{id}$. We first show that $\sigma$ is a homomorphism $S^1\to S^1$ (Claim \ref{claim:hom}), and then that $\sigma$ lifts to a map $\widetilde{\sigma}:\bbR\to\bbR$ (Claim \ref{claim:lift}). We finish the proof by showing that $\widetilde{\sigma}$ is a field automorphism of $\bbR$ and hence trivial (Claims \ref{claim:tildehomog}-\ref{claim:mult}). 

We start with the observation that besides collinearity of points, $f$ preserves another geometric configuration:
	\begin{dfn} A set of 4 points in $T^2$ is a \emph{block} $B$ if the points are the vertices of a square all of whose sides are either horizontal or vertical, i.e. if we can label the points $P,Q,R,S$ such that
		\begin{itemize}
			\item $P$ and $Q$ lie on a horizontal line $h_{x_0}$, and $R$ and $S$ lie on a horizontal line $h_{x_1}$,
			\item $P$ and $R$ lie on a vertical line $v_{x_0}$, and $Q$ and $S$ lie on a vertical line $v_{x_1}$, and
			\item $P$ and $S$ lie on a line of slope 1, and $Q$ and $R$ lie on a line of slope -1.
		\end{itemize}
			\label{dfn:block}
		\end{dfn} 
It is immediate from the fact that $f$ preserves horizontal lines, vertical lines, and lines of slope $\pm 1$, that $f$ preserves blocks. If $B$ is a block, we will denote the block obtained by applying $f$ to the vertices of $B$ by $f(B)$. We have the following useful characterization of blocks:
\begin{lem} Let  $x_0,x_1,y_0,y_1\in S^1$. Then the points $(x_i,y_j)$ where $i,j\in \{1,2\}$ form a block if and only if either
	\begin{enumerate}[(1)]
		\item $x_1-x_0=y_1-y_0$, or
		\item $x_1-x_0=y_0-y_1$.
	\end{enumerate}
	\label{lem:blockcrit}
\end{lem}
\begin{proof} To be done.\end{proof}
We have the following application of this characterization.
\begin{claim} \label{claim:hom} $\sigma:S^1\to S^1$ is a (group) isomorphism. \end{claim}
\begin{proof} Since $\sigma(0)=0$, it suffices to show that $\sigma(x+y)=\sigma(x)+\sigma(y)$ for all $x,y\in S^1$. We will first show that this holds if $y$ is rational. Consider the block $B$ consisting of the vertices $(0,y),(0,x+y), (x,y),$ and $(x,x+y)$. The two alternatives of Lemma \ref{lem:blockcrit} applied to the block $f(B)$ yield
	\begin{enumerate}[(1)]
		\item $\sigma(x)=\sigma(x+y)-\sigma(y),$ or
		\item $\sigma(x)=\sigma(y)-\sigma(x+y).$
	\end{enumerate}
If (1) holds we are done. If (2) holds we reverse the roles of $x,y$ (note that we have not used yet that $y$ is rational). Again we obtain either $\sigma(y)=\sigma(x+y)-\sigma(x)$, in which case we are done, or $\sigma(y)=\sigma(x)-\sigma(x+y)$. If the latter holds, we find that $2(\sigma(x)-\sigma(y))=0$, so that either $\sigma(x)=\sigma(y)$ or $\sigma(x)=\sigma(y)+\frac{1}{2}$. In either case, using that $y=\sigma(y)$ is rational, we see that $\sigma(x)$ is rational as well, and hence we have $x=\sigma(x)$. This finishes the proof under the additional assumption that $y$ is rational.

We will now show that $\sigma(x+y)=\sigma(x)+\sigma(y)$ for all $x,y$. Note that we only used that $y$ is rational in the above argument to deal with the final two cases, where either $\sigma(x)=\sigma(y)$ or $\sigma(x)=\sigma(y)+\frac{1}{2}$. We consider these two cases separately.

\subsection*{Case 1} ($\sigma(x)=\sigma(y)$): Since $\sigma$ is a bijection, we have $x=y$. Consider the block $B$ with vertices $(0,x),(x,x),(0,2x),(x,2x)$. The two alternatives of Lemma \ref{lem:blockcrit} for the block $f(B)$ yield that either
	\begin{enumerate}[(1)]
		\item $\sigma(x)=\sigma(2x)-\sigma(x)$, or
		\item $\sigma(x)=\sigma(x)-\sigma(2x).$
	\end{enumerate}
	In the first case we obtain that $\sigma(2x)=2\sigma(x)$, so that
		$$\sigma(x+y)=\sigma(2x)=2\sigma(x)=\sigma(x)+\sigma(y).$$
	If (2) holds, we have that $\sigma(2x)=0$ and hence $2x=0$, so that $x$ is rational (in which case the claim is already proven).

\subsection*{Case 2} ($\sigma(x)=\sigma(y)+\frac{1}{2}$): In this case we had 
	$$\sigma(x+y)=\sigma(x)-\sigma(y)=\frac{1}{2}.$$ 
Now by using that $\sigma$ is additive if one of the variables is rational, we see that 
	$$\sigma(x)=\sigma(y)+\frac{1}{2}=\sigma\left(y+\frac{1}{2}\right).$$
Since $\sigma$ is a bijection, we must have $x=y+\frac{1}{2}$. Hence
	\begin{align*}
	\sigma(x+y)	&=\sigma\left(2y+\frac{1}{2}\right)=\sigma(2y)+\frac{1}{2}\\
				&=2\sigma(y)+\frac{1}{2}=\sigma\left(y+\frac{1}{2}\right)+\sigma(y)\\
				&=\sigma(x)+\sigma(y),
	\end{align*}
where on the second line we also use that $\sigma(2z)=2\sigma(z)$ for all $z$, which was proven in the solution of Case 1. This completes the proof.\end{proof}

For the remainder of this section, we introduce the following notation: If $\alpha\in \bbR$, let $\ell_\alpha$ be the line in $T^2$ of slope $\alpha$ through 0. Also recall that $h_\alpha$ (resp. $v_\alpha$) denotes the horizontal line $[y=\alpha]$ (resp. the vertical line $[x=\alpha]$). We use the homomorphism property of $\sigma$ to provide a link between the $\sigma$ and the image under $f$ of any line:
	\begin{claim} \label{claim:lift} Let $x_0\in \bbR$ and let $y_0$ be the slope of $f(\ell)$. Then $y_0\equiv \sigma(x_0)\mod \bbZ$.\end{claim}
\begin{proof} Consider the intersections of $\ell$ with the line $v_0:=[x=0]$. We have
		$$\ell\cap v_0=\{[0,k x_0]\mid k\in\bbZ\}.$$
	Let $f(\ell)$ have slope $y_0$. Then $[0,\sigma(x_0)]\in f(\ell)\cap v_0$, so we can write 
		$$\sigma(x_0)=ky_0 \mod \bbZ$$ 
	for some $k\in \bbZ$. On the other hand $[0,y_0]\in f(\ell)\cap v_0=f(\ell\cap v_0)$, so there exists $l\in \bbZ$ such that $y_0\equiv l\sigma(x_0)$ mod $\bbZ$. Hence we have
		$$y_0\equiv l\sigma(x_0)\equiv lk y_0 \mod \bbZ.$$
	Since images under $f$ of lines with irrational slope are irrational, and $x_0$ is irrational, we must have that $y_0$ is irrational. Therefore we have $lk=1$, so that $l=\pm 1$.
	
	It remains to show that $k=1$. Suppose not, so that we have $y_0\equiv -\sigma(x_0)\mod \bbZ$. Write $y_0=-\sigma(x_0)+k_0$ for some $k_0\in \bbZ$. Consider the intersections of $\ell$ with the line $v_{\frac{1}{3}}$. These occur at the points $[\frac{1}{3},\frac{1}{3}x_0+l x_0]$ for $l\in \bbZ$. Now application of $f$ maps these points to intersections of $f(\ell)$ and $v_{\frac{1}{3}}$, which are given by $[\frac{1}{3},\frac{1}{3}y_0+l y_0]$ for $l\in \bbZ$. Consider the image of the point $[\frac{1}{3},\frac{1}{3}x_0]\in \ell\cap v_{\frac{1}{3}}$. Then we can write
		$$\sigma\left(\frac{1}{3}x_0\right)\equiv \frac{1}{3}y_0+k_3 y_0\mod \bbZ$$
	for some $k_3\in \bbZ$. Now using that $y_0=-\sigma(x_0)+k_0$, we have
		$$\sigma\left(\frac{1}{3}x_0\right)\equiv -\frac{1}{3}\sigma(x_0)+\frac{k_0}{3}+k_3\sigma(x_0) \mod \bbZ.$$
	Hence
		$$\sigma(x_0)\equiv 3\sigma\left(\frac{1}{3}x_0\right)\equiv (3k_3-1)\sigma(x_0)\mod \bbZ.$$
	Since $\sigma(x_0)\notin\bbQ$, we must have $3k_3-1=1$, which is a contradiction.
	\end{proof}
	Consider now the map
		$$\widetilde{\sigma}:\bbR\to\bbR$$
	defined by $\widetilde{\sigma}(x_0):=y_0$, where the image under $f$ of a line with slope $x_0$ has slope $y_0$. The claim above establishes that $\widetilde{\sigma}$ is a lift of the map $\sigma:S^1\to S^1$. In the claims below we will establish that $\widetilde{\sigma}$ is a field automorphism of $\bbR$ and hence $\widetilde{\sigma}=\text{id}$.
		\begin{claim} For any integer $a$ and $x\in \bbR$, we have $\widetilde{\sigma}(ax)=a\widetilde{\sigma}(x)$. \label{claim:tildehomog}\end{claim}
		\begin{proof} Let $a\in \bbZ$ and $x_0\in \bbR$. If $x_0$ is rational then the claim is already proven, so we will assume that $x_0$ is irrational. Since we have
			$$\widetilde{\sigma}(ax_0)\equiv \sigma(ax_0)\equiv a\sigma(x_0)\equiv a\widetilde{\sigma}(x_0) \mod \bbZ,$$
		we can choose $k\in\bbZ$ such that $\widetilde{\sigma}(ax_0)=a\widetilde{\sigma}(x_0)+k$. Choose $p>1$ prime and let $\ell_{ax_0}$ be the line in $T^2$ through 0 with slope $ax_0$. The intersections of $\ell_{ax_0}$ and $v_{\frac{1}{p}}$ occur at heights $\frac{a}{p}x_0+a\bbZ x_0$ mod $\bbZ$. Under $f$ these are bijectively mapped to the intersections of $v_{\frac{1}{p}}$ and $\ell_{\widetilde{\sigma}(ax_0)}$, which occur at heights $\frac{1}{p}\widetilde{\sigma}(ax_0)+\bbZ\widetilde{\sigma}(ax_0)$. In particular there exists $l\in \bbZ$ such that
		\begin{equation} \sigma\left(\frac{a}{p}x_0+lax_0\right)\equiv \frac{1}{p}\widetilde{\sigma}(ax_0)\mod \bbZ. \label{eq:relmod} \end{equation}
For the left-hand side we have
		\begin{equation}\sigma\left(\frac{a}{p}x_0+lax_0\right)\equiv a \sigma\left(\frac{1}{p}x_0+lx_0\right)\mod \bbZ.\label{eq:homog}\end{equation}
Note that $[\frac{1}{p},\frac{1}{p}x_0+lx_0]$ is an intersection point of $\ell_{x_0}$ and $v_{\frac{1}{p}}$ so that $[\frac{1}{p},\sigma(\frac{1}{p}x_0+lx_0)]$ is an intersection point of $\ell_{\widetilde{\sigma}(x_0)}$ and $v_{\frac{1}{p}}$. Hence there exists $n\in\bbZ$ such that
	\begin{equation}\sigma\left(\frac{1}{p}x_0+lx_0\right)\equiv \frac{1}{p}\widetilde{\sigma}(x_0)+n\widetilde{\sigma}(x_0)\mod \bbZ.\label{eq:isect}\end{equation}
Combining Equations \ref{eq:relmod}, \ref{eq:homog} and \ref{eq:isect}, we see that
		$$\frac{a}{p}\widetilde{\sigma}(x_0)+an\widetilde{\sigma}(x_0)\equiv \frac{1}{p}\widetilde{\sigma}(ax_0) \mod \bbZ.$$
Further using that $\widetilde{\sigma}(ax_0)=a\widetilde{\sigma}(x_0)+k$, we have
		$$\frac{a}{p}\widetilde{\sigma}(x_0)+an\widetilde{\sigma}(x_0)\equiv \frac{a}{p}\widetilde{\sigma}(x_0)+\frac{k}{p} \mod\bbZ,$$
	so that 
		$$an\widetilde{\sigma}(x_0)\equiv \frac{k}{p}\mod \bbZ.$$
Because $\widetilde{\sigma}(x_0)$ is irrational, this is impossible unless $an=0$ and $\frac{k}{p}\in \bbZ$. Since $p$ was an arbitrary prime number, we must have $k=0$.
		\end{proof}
		\begin{claim} $\widetilde{\sigma}$ is a homomorphism of $\bbR$ (as an additive group).
		\label{claim:tildehom}
		\end{claim}
		\begin{proof} The previous claim with $a=-1$ shows that $\widetilde{\sigma}(-x)=-\widetilde{\sigma}(x)$.  Therefore it remains to show that $\widetilde{\sigma}$ is additive. Let $x,y\in \bbR$ be arbitrary. Since $\widetilde{\sigma}$ is a lift of $\sigma$, we have for any $n\geq 1$:
		\begin{align*}
			\widetilde{\sigma}\left(\frac{x+y}{n}\right)&\equiv \sigma\left(\frac{x}{n}+\frac{y}{n}\right)\\
											&\equiv \sigma\left(\frac{x}{n}\right)+\sigma\left(\frac{y}{n}\right)\\
											&\equiv \widetilde{\sigma}\left(\frac{x}{n}\right)+\widetilde{\sigma}\left(\frac{y}{n}\right)\mod \bbZ.
		\end{align*}
Multiplying by $n$ and using Claim \ref{claim:tildehomog} with $a=n$, we have
			$$\widetilde{\sigma}(x+y)\equiv \widetilde{\sigma}(x)+\widetilde{\sigma}(y)\mod n\bbZ.$$
	Since $n$ was arbitrary, we must have $\widetilde{\sigma}(x+y)=\widetilde{\sigma}(x)+\widetilde{\sigma}(y)$ as desired.	\end{proof}
	
		\begin{claim} For any $x\neq 0$, we have $\widetilde{\sigma}(\frac{1}{x})=\frac{1}{\widetilde{\sigma}(x)}$.\end{claim}
			\begin{proof} We will first show that $\widetilde{\sigma}(\frac{1}{x})\equiv \frac{1}{\widetilde{\sigma}(x)}\mod \bbZ$ for $x\neq 0$. Let $x_0\in \bbR^\times$ and let $\ell_{x_0}$ again be the line with slope $x_0$ through 0. We can assume that $x_0$ is irrational. Note that the intersections of $\ell_{x_0}$ with the horizontal line $h_0$ occur at the points $[\frac{k}{x_0}, 0]$. Under $f$ these are bijectively mapped to the intersections of $\ell_{\widetilde{\sigma}(x_0)}$ with $h_0$. In particular there are integers $k,l\in \bbZ$ such that
			$$\sigma\left(\frac{1}{x_0}\right)\equiv \frac{k}{\widetilde{\sigma}(x_0)}\mod\bbZ$$
		and 
			$$\sigma\left(\frac{l}{x_0}\right)\equiv \frac{1}{\widetilde{\sigma}(x_0)}\mod \bbZ.$$
	Hence 
			$$\frac{1}{\widetilde{\sigma}(x_0)}\equiv \sigma\left(\frac{l}{x_0}\right)\equiv l\sigma\left(\frac{1}{x_0}\right)\equiv kl\frac{1}{\widetilde{\sigma}(x_0)}\mod \bbZ.$$
		Since $\widetilde{\sigma}(x_0)$ is irrational, we must have $kl=1$ so that $k=\pm 1$. We claim that $k=1$. To see this, consider the intersections of $\ell_{x_0}$ with the horizontal line $h_{\frac{1}{3}}$. These occur at $[\frac{1}{3x_0}+\frac{n}{x_0},\frac{1}{3}]$ for $n\in\bbZ$. Under $f$ these are mapped to the intersections of $\ell_{\widetilde{\sigma}(x_0)}$ with $h_{\frac{1}{3}}$. Hence there exists $n\in\bbZ$ such that
		\begin{equation} \sigma\left(\frac{1}{3x_0}\right)\equiv \frac{1}{3\widetilde{\sigma}(x_0)}+\frac{n}{\widetilde{\sigma}(x_0)} \mod \bbZ. \label{eq:rel1}\end{equation}
For the left-hand side, we have
		\begin{equation} \sigma\left(\frac{1}{3x_0}\right)\equiv \widetilde{\sigma}\left(\frac{1}{3x_0}\right)\equiv \frac{1}{3}\widetilde{\sigma}\left(\frac{1}{x_0}\right)\equiv \frac{k}{3}\frac{1}{\widetilde{\sigma}(x_0)} \mod \bbZ,\label{eq:rel2}\end{equation}
	where we used that $\widetilde{\sigma}$ is $\bbQ$-linear (because it is a homomorphism $\bbR\to\bbR$). Combining Equations \ref{eq:rel1} and \ref{eq:rel2}, we find
		$$(3n+1-k)\frac{1}{\widetilde{\sigma}(x_0)}\equiv 0\mod 3\bbZ.$$
	Since $\widetilde{\sigma}(x_0)$ is not rational, we must have $3n+1-k=0$ so that $k\equiv 1\mod 3$. Since we already found that $k=1$ or $-1$, we must have $k=1$, as desired.
	
	At this point we know that $\widetilde{\sigma}(\frac{1}{x})\equiv \frac{1}{\widetilde{\sigma}(x)}\mod \bbZ$ for any $x\neq 0$. It remains to show that $\widetilde{\sigma}(\frac{1}{x})=\frac{1}{\widetilde{\sigma}(x)}$. To see this, let $x\in \bbR^\times$ be arbitrary. We can assume that $x$ is irrational. Let $n\in \bbZ$ such that 
		$$\widetilde{\sigma}\left(\frac{1}{x}\right)=\frac{1}{\widetilde{\sigma}(x)}+n.$$
	Set $N=|n|+1$, and note that
		$$\widetilde{\sigma}\left(\frac{1}{Nx}\right)=\frac{1}{N}\widetilde{\sigma}\left(\frac{1}{x}\right)=\frac{1}{\widetilde{\sigma}(NX)}+\frac{n}{N}.$$
	Hence if $n\neq 0$, we have $\widetilde{\sigma}\left(\frac{1}{Nx}\right)\not\equiv \frac{1}{\widetilde{\sigma}(Nx)}\mod \bbZ$, which is a contradiction.\end{proof}
In the final two claim we will show that $\widetilde{\sigma}$ is multiplicative.
	\begin{claim} For any $x,y\in \bbR$, we have $\widetilde{\sigma}(xy)=\widetilde{\sigma}(x)\widetilde{\sigma}(y)$. \label{claim:mult}	\end{claim}
	\begin{proof} Let $x_0,y_0\in \bbR$. Note that since $\widetilde{\sigma}$ is an isomorphism of $\bbR$ as an additive group, $\widetilde{\sigma}$ is $\bbQ$-linear. Hence without loss of generality, we assume that $x_0$ is irrational. Consider the intersections of $\ell_{x_0}$ with the vertical line $v_{y_0}$. These occur at the points $[y_0,x_0 y_0+k x_0]$ for $k\in \bbZ$. Under $f$ these intersection points are mapped to the points of intersection of $\ell_{\widetilde{\sigma}(x_0)}$ with $v_{\widetilde{\sigma}(y_0)}$. In particular, by considering the image of $[y_0,x_0 y_0]$, we see that there is $k\in \bbZ$ such that
		$$\sigma(x_0 y_0)\equiv \widetilde{\sigma}(x_0)\widetilde{\sigma}(y_0)+k\widetilde{\sigma}(x_0)\mod \bbZ.$$
	Since $\sigma(x_0 y_0)\equiv \widetilde{\sigma}(x_0 y_0)\mod \bbZ$, we see that
		$$\widetilde{\sigma}(x_0 y_0)-\widetilde{\sigma}(x_0)\widetilde{\sigma}(y_0)\in \bbZ+\bbZ\widetilde{\sigma}(x_0).$$
	For $x\in \bbR$, let $\mu_x:\bbR\to \bbR$ denote multiplication by $x$. Then the above argument shows that the homomorphism
		$$(\widetilde{\sigma}\circ  \mu_{x_0})-\left(\mu_{\widetilde{\sigma}(x_0)}\circ \widetilde{\sigma}\right):\bbR \to \bbR$$
	has image contained in $\bbZ+\bbZ\widetilde{\sigma}(x_0)$. Since $\widetilde{\sigma}(x_0)$ is irrational, we have $\bbZ+\bbZ\widetilde{\sigma}(x_0)\cong \bbZ^2$ as additive groups. But any homomorphism $\bbR\to \bbZ^2$ is trivial (because $\bbR$ is divisible), so that $\widetilde{\sigma}\circ \mu_{x_0}=\mu_{\widetilde{\sigma}(x_0)}\circ \widetilde{\sigma}$, as desired.\end{proof}
As commented at the beginning of this section, the above results finish the 2-dimensional case of Theorem \ref{thm:main}. Indeed, up to precomposition by an affine automorphism, any bijection $f:T^2\to T^2$ that maps lines to lines, is of the form $(\sigma,\sigma)$, where $\sigma$ is a homomorphism $S^1\to S^1$ that lifts to a field automorphism $\widetilde{\sigma}$ of $\bbR$. Since any field automorphism of $\bbR$ is trivial, it follows that $\sigma=\text{id}$ and hence the original map is affine.

\section{The $n$-dimensional case}
\label{sec:ndim}
We finish the proof of Theorem \ref{thm:main} that any bijection of $T=\bbR^n\slash \bbZ^n, \, n\geq 2,$ that maps lines to lines, is an affine map. We argue by induction on $n$. The base case $n=2$ has been proven in the previous section. Let $f:T\to T$ be a bijection that maps lines to lines. We recall that in Section \ref{sec:planes}, we showed that for any rational subtorus $S\subseteq T$, the image $f(S)$ is also a rational subtorus.

\begin{proof}[Proof of Theorem \ref{thm:main}] Without loss of generality, we assume $f(0)=0$. For $1\leq i\leq n$, we let 
	$$\ell_i:=\begin{bmatrix} \bbR e_i \end{bmatrix}$$
denote the (image of the) coordinate line. Each $\ell_i$ is a rational line, so $f(\ell_i)$ is also a rational line. Choose primitive integral vectors $v_i\in\bbZ^n$ such that
	$$f(\ell_i)=\begin{bmatrix} \bbR v_i\end{bmatrix}.$$
Note that $v_i$ is unique up to sign. Let $A$ be the linear map of $\bbR^n$ with $A e_i = v_i$ for every $i$. Our goal is to show that $f=A$ as maps of the torus. 
\begin{claim} $A$ is invertible with integer inverse. \label{claim:ainv} \end{claim}
\begin{proof} To show that $A$ is invertible, we need to show that $v_1,\dots, v_n$ span $\bbR^n$. To show that $A^{-1}$ has integer entries, we need to show that in addition, $v_1,\dots,v_n$ generate $\bbZ^n$ (as a group). For $1\leq j\leq n$, set
	$$U_j:=\Span\{e_i \mid i\leq j\},$$
and
	$$V_j:=\Span\{v_i \mid i \leq j\}.$$
Note that $f [U_j]=[V_j]$ by Lemma \ref{lem:span}. Taking $j=n$, this already shows that $v_1,\dots,v_n$ span $\bbR^n$, so $A$ is invertible. We argue by induction on $j$ that $\{v_1,\dots,v_j\}$ generate $\pi_1[V_j]\subseteq \pi_1 T$. For $j=n$, this exactly means that $v_1,\dots,v_n$ generate $\pi_1 T$, which would finish the proof.

The base case $j=1$ is exactly the assertion that $v_1$ is a primitive vector. Now suppose the statement is true for some $j$. Consider the composition
	\begin{equation} f(\ell_j) \hookrightarrow V_j \to V_j\slash V_{j-1}. \label{eq:quot} \end{equation}
This composition is a homomorphism of  $f(\ell_j)\cong S^1$ to $V_j\slash V_{j-1}\cong S^1$. The kernel is given by
	$$f(\ell_j)\cap V_{j-1}= f(\ell_j \cap U_{j-1})=f(0)=0.$$
Therefore the map $f(\ell_j)\rightarrow V_j\slash V_{j-1}$ is an isomorphism, and 
	$$\pi_1(V_j)\cong \pi_1(V_{j-1})\oplus \pi_1(f(\ell_j)).$$
Since we know by the inductive hypothesis that $\pi_1(V_{j-1})$ is generated by $v_1,\dots,v_{j-1}$, and that $\pi_1(f(\ell_j))$ is generated by $v_j$ (again because $v_j$ is primitive), we find that $\pi_1(V_j)$ is generated by $v_1,\dots,v_j$, as desired.\end{proof}

For the remainder of the proof we set $g:=A^{-1}\circ f$. Our goal is to show that $g=\text{id}$. Note that $g$ is a bijection of $T$ with $g(0)=0$ and $g(\ell_i)=\ell_i$ for every $i$. Let 
	$$H_{i}:=\begin{bmatrix} \Span(e_j \mid j\neq i)\end{bmatrix}$$ 
denote the (image of the) coordinate hyperplane. Since $g(\ell_i)=\ell_i$ for every $i$, we know (using Lemma \ref{lem:span} again) that $g(H_i)=H_i$. By the inductive hypothesis, $g$ is given by a linear map $A_i$ on $H_i$. But any linear automorphism of $H_i$ that leaves the coordinate lines $\ell_j, \, j\neq i,$ invariant, must be of the form
	$$\begin{pmatrix} \pm 1 & & \\ & \ddots & \\ & & \pm 1 \end{pmatrix}.$$
Now recall that $v_i$ was unique up to sign. Replace $v_i$ by $-v_i$ whenever $g|_{\ell_i} = -\text{id}$. After this modification, we have that $g=\text{id}$ on every coordinate hyperplane $H_i$. Our goal is to show that $g=\text{id}$ on $T$.

\begin{claim} Fix $1\leq i\leq n$ and let $\ell$ be any line parallel to $\ell_i$. Then $g(\ell)=\ell$.\end{claim}
\begin{proof} First recall the following elementary general fact about the number of intersections of a line with a coordinate hyperplane: If $m$ is any rational line and
	$$v=\begin{pmatrix} v_1 \\ \vdots \\ v_n\end{pmatrix}$$
is a primitive integer vector with $m=[\bbR v]$, then
	$$|v\cap H_i|=|v_i|.$$
Now let $\ell$ be parallel to $\ell_i$. Then
	$$|\ell\cap H_j|=\delta_{ij},$$
and hence also
	$$|g(\ell)\cap H_j|=|g(\ell)\cap g(H_j)|=|\ell\cap H_j|=\delta_{ij}.$$
Therefore we see that if $v$ is a primitive integer vector with $g(\ell)=[\bbR v]$, then $v=e_i$. This exactly means that $\ell$ and $g(\ell)$ are parallel. Hence to show that $\ell=g(\ell)$, it suffices to show that $\ell\cap g(\ell)$ is nonempty. 

Since $\ell$ is parallel to $\ell_i$, there is a unique point of intersection $x_\ell:=\ell\cap H_i$. Since $x_\ell \in H_i$, we have $g(x_\ell)=x_\ell$, so $x_\ell\in \ell\cap g(\ell)$, as desired.\end{proof}

We are now able to finish the proof of the main theorem. Let $x\in T$ and let $\ell_i(x)$ be the unique line parallel to $\ell_i$ that passes through $x$. For any two distinct indices $i\neq j$, the point $x$ is the unique point of intersection of $\ell_i(x)$ and $\ell_j(x)$. On the other hand,
	$$g(x)\in g(\ell_i(x))\cap g(\ell_j(x))=\ell_i(x)\cap \ell_j(x),$$
so we must have $g(x)=x$.\end{proof}

\bibliographystyle{alpha}
\bibliography{ftagref}

\begin{thebibliography}{Mat07}

\bibitem[Ale67]{ftlg}
A.~D. Alexandrov.
\newblock A contribution to chronogeometry.
\newblock {\em Canad. J. Math.}, 19:1119--1128, 1967.

\bibitem[Art88]{ftpg}
E.~Artin.
\newblock {\em Geometric algebra}.
\newblock Wiley Classics Library. John Wiley \& Sons, Inc., New York, 1988.
\newblock Reprint of the 1957 original, A Wiley-Interscience Publication.

\bibitem[CP99]{ftgnoninv}
Alexander Chubarev and Iosif Pinelis.
\newblock Fundamental theorem of geometry without the {$1$}-to-{$1$}
  assumption.
\newblock {\em Proc. Amer. Math. Soc.}, 127(9):2735--2744, 1999.

\bibitem[Gol16]{geomstruct}
William Goldman.
\newblock Geometric structures on manifolds.
\newblock Lecture notes, available at
  \url{http://www.math.umd.edu/~wmg/gstom.pdf}, 2016.

\bibitem[Jef00]{fthg}
Jason Jeffers.
\newblock Lost theorems of geometry.
\newblock {\em Amer. Math. Monthly}, 107(9):800--812, 2000.

\bibitem[Mat03]{geodeq}
Vladimir~S. Matveev.
\newblock Hyperbolic manifolds are geodesically rigid.
\newblock {\em Invent. Math.}, 151(3):579--609, 2003.

\bibitem[Mat07]{PLOC}
Vladimir~S. Matveev.
\newblock Proof of the projective {L}ichnerowicz-{O}bata conjecture.
\newblock {\em J. Differential Geom.}, 75(3):459--502, 2007.

\bibitem[Sta47]{ftag}
G.~Von Staudt.
\newblock {\em Geometrie der Lage}.
\newblock N\"urnberg: F. Korn, 1847.

\bibitem[Zeg16]{zeghibPLC}
Abdelghani Zeghib.
\newblock On discrete projective transformation groups of {R}iemannian
  manifolds.
\newblock {\em Adv. Math.}, 297:26--53, 2016.

\end{thebibliography}

\end{document}